\documentclass[11pt]{article}

\usepackage{amsmath,amsthm}

\usepackage{amssymb,latexsym}

\usepackage{enumerate}

\newcommand{\BB}{{\cal B}}

\newcommand{\EE}{{\cal E}}
\newcommand{\FF}{{\cal F}}

\newcommand{\MM}{{\cal M}}

\newcommand{\BM}{{\mathbb M}}
\newcommand{\BN}{{\mathbb N}}
\newcommand{\BR}{{\mathbb R}}

\newcommand{\BBM}{{\mathbf M}}

\newcommand{\fch}{{\mathbf{1}}}

\newtheorem{theorem}{\bf Theorem}[section]
\newtheorem{proposition}[theorem]{\bf Proposition}
\newtheorem{lemma}[theorem]{\bf Lemma}

\theoremstyle{definition}
\newtheorem{definition}[theorem]{Definition}

\newtheorem{remark}[theorem]{Remark}

\numberwithin{equation}{section}

\begin{document}

\title {Renormalized solutions of semilinear elliptic equations
with general measure data}
\author {Tomasz Klimsiak and Andrzej Rozkosz}
\date{}
\maketitle
\begin{abstract}
In the paper, we first propose a definition of renormalized
solution of semilinear elliptic equation involving operator
corresponding to a general (possibly nonlocal) symmetric regular
Dirichlet form satisfying the so-called absolute continuity
condition and general (possibly nonsmooth) measure data. Then we
analyze the relationship between our definition and other concepts
of solutions considered in the literature (probabilistic
solutions, solution defined via the resolvent kernel of the
underlying Dirichlet form, Stampacchia's definition by duality).
We show that under mild integrability assumption on the data all
these concepts coincide.
\end{abstract}

\footnotetext{T. Klimsiak: Institute of Mathematics, Polish
Academy of Sciences, \'Sniadeckich 8, 00-956 Warszawa, Poland, and
Faculty of Mathematics and Computer Science, Nicolaus Copernicus
University, Chopina 12/18, 87-100 Toru\'n, Poland. E-mail:
tomas@mat.umk.pl.}

\footnotetext{A. Rozkosz: Faculty of
Mathematics and Computer Science, Nicolaus Copernicus University,
Chopina 12/18, 87-100 Toru\'n, Poland. E-mail: rozkosz@mat.umk.pl.}

\section{Introduction}
\label{sec1}

Let $L$ be the operator associated with a symmetric regular
Dirichlet form $(\EE,D(\EE))$ on $L^2(E;m)$,
$f:E\times\BR\rightarrow\BR$ be a measurable function and $\mu$ be
a bounded signed Borel measure on $E$. In the paper we consider
semilinear equations of the form
\begin{equation}
\label{eq1.1} -Lu=f(\cdot,u)+\mu\quad\mbox{in }E.
\end{equation}
One of the important problems that arises when studying such
equations is the problem of proper definition of a solution. This
problem has been dealt  with by many authors. In the present paper
we first introduce  yet another definition of a solution of
(\ref{eq1.1}). It is a  slight modification of the definition of a
renormalized solution introduced in \cite{KR:NoD} in case $\mu$ is
smooth. Then we analyze the relationship between this new
definition and other concepts of solutions known in the
literature.

In case $L$ is a uniformly elliptic divergence form operator and
$f$ does not depend on $u$, some definition, now called
Stampacchia's definition by duality, was proposed by Stampacchia
\cite{S} in 1965. Later on, to deal with equations with more
general local operator $L$, the definitions of entropy solution
and renormalized solution were introduced. For a comparison of
different forms of these definitions and remarks on other concepts
of solutions of equations of the form (\ref{eq1.1}) with local
operator $L$ and $f$ not depending on $u$ see \cite{DMOP}. Elliptic equations
with local operators and nonlinear dependence on general measure data are studied in \cite{DDO,MP}.

In case $f$ depends on $u$ most of known results are devoted to
the case where $\mu$ is smooth. Recall (see \cite{FST}) that $\mu$
admits a unique decomposition
\begin{equation}
\label{eq1.2} \mu=\mu_d+\mu_c
\end{equation}
into the smooth (diffuse) part $\mu_d$ and the concentrated part
$\mu_c$, i.e. $\mu_d$ is a bounded Borel measure, which is
``absolutely continuous" with respect to the capacity $\mbox{Cap}$
determined by $(\EE,D(\EE))$, and $\mu_c$ is a bounded Borel
measure which is ``singular" with respect to $\mbox{Cap}$. In case
$L$ is local and $\mu$ is smooth entropy and renormalized
solutions of (\ref{eq1.1}) are studied in numerous papers (see,
e.g., \cite{BBGGPV,DP} and the references given there). A
definition of renormalized solutions  applicable to (\ref{eq1.1})
with  general  $L$ associated with a general transient (possibly
non-symmetric) Dirichlet form was recently given in \cite{KR:NoD}.
If $(\EE,D(\EE))$ is symmetric and $f(\cdot,u)\in L^1(E;m)$,
renormalized solutions in the sense of \cite{KR:NoD} coincide with
probabilistic solutions of (\ref{eq1.1}) defined earlier in
\cite{KR:JFA} (see also \cite{KR:CM} for
equations with operator $L$ associated with a non-symmetric
quasi-regular form and \cite{KR:PA} for equations with nonlinear dependence on measure data). Recall that a measurable $u:E\rightarrow\BR$
is a probabilistic solution of (\ref{eq1.1}) in the sense of
\cite{KR:JFA,KR:CM} if the following nonlinear Feynman-Kac formula
\begin{equation}
\label{eq1.3} u(x)=E_x\Big(\int^{\zeta}_0f(X_t,u(X_t))\,dt
+\int^{\zeta}_0dA^{\mu}_t\Big)
\end{equation}
is satisfied for quasi-every $x\in E$. In (\ref{eq1.3}),
$\BBM=(X,P_x)$ is a Markov process with life time $\zeta$
associated with $\EE$, $E_x$ denotes the expectation with respect
to $P_x$ and $A^{\mu}$ is the continuous additive functional of
$\BBM$ associated with $\mu$ in the Revuz sense (see Section
\ref{sec2}). The equivalence between renormalized and
probabilistic solutions allows one to use effectively
probabilistic methods in the study of renormalized solutions of
(\ref{eq1.1}). Also note that if $f\in L^1(E;m)$ then renormalized
solutions of (\ref{eq1.1}) coincide with Stampacchia's solutions
by duality defined in \cite{KR:JFA,KR:CM}.

The semilinear case with general, possibly nonsmooth bounded
measure $\mu$ is much more involved. The study of (\ref{eq1.1})
with nonsmooth measure was initiated in 1975 by Brezis and
B\'enilan in case $L$ is the Laplace operator $\Delta$ (see
\cite{BB,BMP} and the references given there for results and
historical comments). For some existence and uniqueness results in
case $L$  is the fractional Laplacian $\Delta^{\alpha/2}$ with
$\alpha\in(0,2)$ see Chen and V\'eron \cite{CV}. Very recently,
Klimsiak \cite{K:CVPDE} started the study of (\ref{eq1.1}) in case
$L$ corresponds to  a transient symmetric regular Dirichlet form
satisfying  the following absolute continuity condition:
\begin{description}
\item[\rm(ACR)]$R_{\alpha}(x,\cdot)$ is absolutely continuous with
respect to $m$ for each $\alpha>0$ and $x\in E$,
\end{description}
where $R_{\alpha}(x,dy)$ denotes the resolvent kernel associated
with $(\EE,D(\EE))$ (see Section \ref{sec2.2}). Equivalently,
\begin{description}
\item[\rm(ACT)]$p_t(x,\cdot)$ is absolutely continuous with
respect to $m$ for each $t>0$ and $x\in E$,
\end{description}
where $p_t(x,dy)$ is the transition function associated with
$(\EE,D(\EE))$. The above conditions are satisfied for instance if
$L$ is a uniformly divergence form operator or
$L=\Delta^{\alpha/2}$ with $\alpha\in(0,2)$. If the form is
transient, then under (ACR) the resolvent kernel $R_0(x,dy)$ has a
density $r$. In \cite{K:CVPDE} a measurable function $u$ on $E$ is
called a solution of (\ref{eq1.1}) if
\begin{equation}
\label{eq1.4} u(x)=\int_E r(x,y)f(y,u(y))\,dy+\int_E
r(x,y)\,\mu(dy)
\end{equation}
for quasi every $x\in E$. In case $\mu_c=0$, the above equation
reduces to (\ref{eq1.3}), so the definition of \cite{K:CVPDE}
reduces to the probabilistic definition of a solution given in
\cite{KR:JFA,KR:CM}. In \cite{K:CVPDE} also a partly probabilistic
interpretation of (\ref{eq1.4}) is given. This suggests that
solutions defined via the resolvent density, i.e. by
(\ref{eq1.4}), may be equivalently defined as renormalized
solutions in the same manner as in \cite{KR:NoD}. In the present
paper we show that this is indeed possible. The definition of a
renormalized solution adopted in the present paper is a minor
modification of the definition of \cite{KR:NoD}. In our opinion,
it is natural, especially from the probabilistic point of view.
Moreover, in many cases considered so far in the literature ($\mu$
is smooth or $\mu$ is nonsmooth and $L=\Delta$ or
$L=\Delta^{\alpha/2}$, like in \cite{BMP,CV})  the solutions
considered there coincide with the renormalized defined in the
present paper.

The main result of the paper says that if the form is transient
and (ACR) is satisfied then the renormalized solution is a
solution in the sense of (\ref{eq1.4}), and if $u$ is a solution
of (\ref{eq1.1}) in the sense of (\ref{eq1.4}) and $u\in L^1(E;m)$
then $u$ is a renormalized solution. We find important that, as in
the case of smooth measures, this correspondence when combined
with probabilistic interpretation of (\ref{eq1.4}) given in
\cite{K:CVPDE} enables one to study renormalized solutions of
(\ref{eq1.1}) with the help of probabilistic methods. For results
on (\ref{eq1.1}) obtained in this way we defer the reader to
\cite{K:CVPDE}). Finally, note that at the end of the paper we
describe some interesting situations in which solutions of
(\ref{eq1.1}) in the sense of (\ref{eq1.4}) automatically have the
property that $f(\cdot,u)\in L^1(E;m)$.

\section{Preliminaries}
\label{sec2}

In the paper  $E$ is  a separable locally compact metric space and
$m$ is a Radon measure on $E$ such that supp$[m]=E$. By $\BB(E)$
(resp. $\BB^+(E)$) we denote the set of all real (resp.
nonnegative) Borel measurable functions on $E$, and by $\BB_b(E)$
the subset of $\BB(E)$ consisting of all bounded functions.

For $u:E\rightarrow\BR$ we set $u^+(x)=\max\{u(x),0\}$,
$u^-(x)=\max\{-u(x),0\}$.

\subsection{Dirichlet forms}

By $(\EE,D(\EE))$ we denote a symmetric regular Dirichlet form on
$H=L^2(E;m)$ (see \cite[Section 1.1]{FOT} for the definition). In
case  $(\EE,D(\EE))$ is transient,  by $(D_e(\EE),\EE)$ we denote
the  extended Dirichlet space of $(\EE,D(\EE))$ (see \cite[Section
1.5]{FOT}).

In the paper, we define capacity $\mbox{Cap}$ as in \cite[Section
2.1]{FOT}. Recall that an increasing sequence $\{F_n\}$ of closed
subsets of $E$ is called nest if Cap$(E\setminus F_n)\rightarrow
0$ as $n\rightarrow \infty$. A subset $N\subset E$ is called
exceptional if Cap$(N)=0$. We will say that some property of
points in $E$ holds quasi everywhere (q.e. for short) if the set
for which it does not hold is exceptional.

We say that a function $u$ on $E$ is quasi-continuous if there
exists a nest $\{F_n\}$ such that $u_{|F_n}$ is continuous for
every $n\ge 1$. By \cite[Theorem 2.1.7]{FOT}, each function $u\in
D_e(\EE)$ has a quasi-continuous $m$-version.

Let $\mu$ be a signed Borel measure on $E$, and let
$|\mu|=\mu^{+}+\mu^-$, where $\mu^+$ (resp. $\mu^-$) we denote the
positive (resp. negative) part of 
of $\mu$. We say that $\mu$ is smooth if $|\mu|$ does not charge
exceptional sets and there exists a nest $\{F_n\}$ such that
$|\mu|(F_n)<\infty$, $n\ge 1$. The set of all smooth measures on
$E$ will be denoted by $S$. By $\MM_b$ we denote the set of all
signed Borel measures on $E$ such that
$\|\mu\|_{TV}:=|\mu|(E)<\infty$, and by $\MM_{0,b}$ the subset of
$\MM_b$ consisting of all smooth measures. $S^+$ is the subset of
$S$ consisting of nonnegative measures. Similarly we define
$\MM^+_b,\MM^+_{0,b}$. By \cite[Lemma 2.1]{FST}, for every
$\mu\in\MM_b$ there exists a unique pair
$(\mu_d,\mu_c)\in\MM_b\times\MM_b$ such that $\mu_d\in\MM_{0,b}$,
$\mu_c$ is concentrated on some exceptional Borel subset of $E$
and (\ref{eq1.2}) is satisfied. If $\mu$ is nonnegative, so are
$\mu_d,\mu_c$. For a complete description of the structure of
$\mu_c$ see \cite{KR:PAMS}.

\subsection{Markov processes}
\label{sec2.2}

Let $E\cup\Delta$ be the one-point compactification of $E$. When
$E$ is already compact, we adjoin $\Delta$ to $E$ as an
isolated point. We adopt the convention that every function $f$ on
$E$ is extended to $E\cup\{\Delta\}$ by setting $f(\Delta)=0$.

By \cite[Theorems 4.2.8, 7.2.3]{FOT} there exists a unique (up to
equivalence) $m$-symmetric Hunt process
$\BBM=(\Omega,\FF,(\FF_t)_{t\ge0},(X_t)_{t\ge0},\zeta, (P_x)_{x\in
E\cup\Delta})$ with state space $E$, life time $\zeta$ and
cemetery state $\Delta$ whose Dirichlet space is $(\EE,D(\EE))$.
This means in particular  that for every $\alpha>0$ and $f\in
\BB_b(E)\cap H$ the resolvent of $\BBM$, that is the function
\[
R_{\alpha}f(x)=E_x\int^{\infty}_0e^{-\alpha t}f(X_t)\,dt,\quad
x\in E
\]
is a quasi-continuous $m$-version of $G_{\alpha}f$.

Let $R_{\alpha}(x,dy)$ denote the kernel on $(E,\BB(E))$ defined
as $R_{\alpha}(x,B)=R_{\alpha}\fch_B(x)$. In the paper we will
assume that $\BBM$ satisfies (ACR) condition formulated in Section
\ref{sec1}. By \cite[Theorem 4.2.4]{FOT}, for symetric forms
considered in the present paper (ACR) is equivalent to (ACT). In
general, for non-symmetric forms, (ACT) is stronger than (ACR).
Also note that in the literature (ACR) is sometimes  called
Meyer's hypothesis (L) (see \cite[Chapter I, Exercise 10.25]{Sh}

Assume that $(\EE,D(\EE))$ is transient. Then there exists a
nonnegative $\BB(E)\otimes\BB(E)$-measurable function $r:E\times
E\rightarrow\BR$ such that $r(x,y)=r(y,x)$, $x,y\in E$ and for
every Borel set $B\subset E$,
\[
R(x,B)=\int_Br(x,y)\,m(dy),\quad x\in E.
\]
In fact, $r(x,y)=\lim_{\alpha\downarrow0}r_{\alpha}(x,y)$, where
$r_{\alpha}(x,y)$ is the density of $R_{\alpha}(x,dy)$ constructed
in \cite[Lemma 4.2.4]{FOT} (see remarks in \cite[p. 256]{BG}). We
call $r$ the resolvent density.

In what follows given a positive Borel measure on $E$, we write
\[
R_{\alpha}\mu(x)=\int_Er_{\alpha}(x,y)\,\mu(dy),\qquad R\mu(x)=\int_Er(x,y)\,\mu(dy),\quad x\in E,\quad\alpha>0.
\]
For a signed Borel measure $\mu$ on $E$, we set
$R\mu(x)=R\mu^+(x)-R\mu^-(x)$, whenever $R\mu^+(x)<+\infty$ or
$R\mu^-(x)<+\infty$, and we adopt the convention that
$R\mu(x)=+\infty$ if $R\mu^+(x)=R\mu^-(x)=+\infty$.

\begin{proposition}
\label{prop2.1} Assume that $(\EE,D(\EE))$ is transient and
\mbox{\rm(ACR)} is satisfied. If $\mu\in\MM_b$ then
$R|\mu|(x)<+\infty$ for q.e. $x\in E$.
\end{proposition}
\begin{proof}
See \cite[Proposition 3.2]{K:CVPDE}.
\end{proof}

Denote by $\BM$ the set of all signed Borel measures $\mu$  on $E$
such that $R|\mu|(x)<+\infty$ for $m$-a.e. $x\in E$. By
Proposition \ref{prop2.1}, $\MM_b\subset\BM$. In general, the
inclusion is strict (see the remark following \cite[Proposition
3.2]{KR:CM}).

We define additive functional (AF  in abbreviation) and continuous AF of $\BBM$
as in \cite[Sections 5.1]{FOT}. By \cite[Theorem 5.1.4]{FOT},
there is a one to one correspondence
(called Revuz correspondence) between the set of smooth  measures
$\mu$ on $E$ and the set of  positive continuous AFs $A$ of $\BBM$. It is given by the
relation
\[
\lim_{t\rightarrow0^+}\frac1{t}E_m\int^t_0f(X_s)\,dA_s=\int_Ef(x)\,\mu(dx),
\quad f\in\BB^+(E),
\]
where $E_m$ denotes the expectation with respect to the measure
$P_m(\cdot)=\int_EP_x(\cdot)\,m(dx)$. In what follows the positive
continuous AF of $\BBM$ corresponding to a positive $\mu\in S$ will be denoted
by $A^{\mu}$. If $\mu$ in $S$, then $\mu^+,\mu^-\in S$, and we set
$A^{\mu}=A^{\mu^+}-A^{\mu^-}$. Note that if $\mu\in S^+$ then for
every $\alpha\ge0$,
\begin{equation}
\label{eq2.1} R_{\alpha}\mu(x)=E_x\int^{\zeta}_0e^{-\alpha t}\,d
A^{\mu}_t=E_x\int^{\infty}_0e^{-\alpha t}\,d A^{\mu}_t
\end{equation}
for q.e. $x\in E$. Indeed, if $\alpha>0$ and $\mu$ is a measure of
finite 0-order energy integral ($\mu\in S^{(0)}_0$ in notation;
see \cite[Section 2.2]{FOT} for the definition), then
(\ref{eq2.1}) follows from Exercise 4.2.2 and Lemma 5.1.3 in
\cite{FOT}. The general case follows by approximation. We first
let $\alpha\downarrow0$ to get (\ref{eq2.1}) for $\alpha\ge0$ and
$\mu\in S^{(0)}_0$, and then we use the 0-order version of
\cite[Theorem 2.2.4]{FOT} (see remark following \cite[Corollary
2.2.2]{FOT}) to get (\ref{eq2.1}) for any $\alpha\ge0$ and $\mu\in
S^+$.

\section{Probabilistic solutions and solutions defined via the
resolvent density}
\label{sec3}

We assume that $(\EE,D(\EE))$ is transient and \mbox{\rm(ACR)} is
satisfied. Consider the problem
\begin{equation}
\label{eq3.1} -Lu=f_u+\mu,
\end{equation}
where $f:E\times\BR\rightarrow\BR$ is a measurable function,
$f_u=f(\cdot,u)$, $\mu\in\BM$ and $L$ is the operator associated
with $(\EE,D(\EE))$, i.e. the nonpositive definite self-adjoint
operator on $H$ such that
\[
D(L)\subset D(\EE),\qquad \EE(u,v)=(-Lu,v),\quad u\in D(L),v\in
D(\EE),
\]
where $(\cdot,\cdot)$ denotes the usual inner product in $H$ (see
\cite[Corollary 1.3.1]{FOT}).

The following two definitions of solutions of (\ref{eq3.1}) were
introduced in \cite{K:CVPDE}.

\begin{definition}
\label{def3.1} We say that a measurable  function
$u:E\rightarrow\BR\cup\{-\infty,+\infty\}$ is a solution of
(\ref{eq1.1}) if $f_u\cdot m\in\BM$ and (\ref{eq1.4}) is satisfied
for q.e. $x\in E$.
\end{definition}

\begin{definition}
\label{def3.2} We say that a measurable
$u:E\rightarrow\BR\cup\{-\infty,+\infty\}$ is a probabilistic
solution of (\ref{eq1.1}) if
\begin{enumerate}
\item[(a)] $f_u\cdot m\in\BM$ and there exists an AF $M$ of $\BBM$ such that such
that for q.e. $x\in E$ the process $M$ is an $(\FF)_{t\ge0}$-local martingale under $P_x$ and
\begin{equation}
\label{eq2.2}  u(X_t)=u(X_0)-\int_0^tf_u(X_s)\,ds
-\int^t_0dA^{\mu_d}_s+\int_0^t\,dM_s, \quad t\ge 0,\quad
P_x\mbox{-a.s.}
\end{equation}
\item[(b)] for every exceptional set $N\subset E$, every stopping time
$T$ such that $T\ge\zeta$ and  every sequence
$\{\tau_k\}\subset\mathcal{T}$ such that $\tau_k\nearrow T$ and
$E_x\sup_{t\le \tau_k}|u(X_t)|<\infty$ for all $x\in E\setminus N$
and $k\ge 1$, we have
\begin{equation}
\label{eq3.3} E_xu(X_{\tau_k})\rightarrow R\mu_c(x),\quad x\in
E\setminus N.
\end{equation}
\end{enumerate}

Any sequence $\{\tau_k\}$ with the  properties listed in condition
(b) will be called the reducing sequence for $u$, and we will say
that $\{\tau_k\}$ reduces $u$.
\end{definition}

\begin{remark}
(i) By \cite[Renark 3.10]{K:CVPDE}, if $\mu_c=0$, then
the above  definition reduces to the definition introduced in \cite{KR:JFA}.
\smallskip\\
(ii) Assume that $u$ is a probabilistic solution of (\ref{eq1.1}). Then
for q.e. $x\in E$ we have
\begin{equation}
\label{eq3.4} E_xu^+(X_{\tau_k})\rightarrow R\mu^+_c(x),\qquad
E_xu^-(X_{\tau_k})\rightarrow R\mu^-_c(x).
\end{equation}
Indeed, if $u$ is a solution of (\ref{eq1.1}) then by
\cite[Theorem 6.3]{K:CVPDE}, $Lu^+\in\BM$. In different words,
$u^+$ is a solution of the equation $Lu^+=\nu$ with some $
\nu\in\BM$. Hence, by condition (b) of Definition \ref{def3.2},
$E_xu^+(X_{\tau_k})\rightarrow R\nu_c(x)$ for q.e. $x\in E$. But
by \cite[Theorem 6.3]{K:CVPDE}, $(Lu^+)_c=(Lu)^+_c$. Hence
$\nu_c=(f_u\cdot m+\mu)^+_c=\mu^+_c$, which proves the first
convergence in (\ref{eq3.4}). The second convergence follows from
the first one and (\ref{eq3.3}).
\end{remark}

\begin{proposition}
\label{prop3.4} Let $\mu\in\BM$. A measurable
$u:E\rightarrow\BR\cup\{-\infty,+\infty\}$ is a solution of
\mbox{\rm(\ref{eq1.1})} in the sense of Definition
\mbox{\rm\ref{def3.1}} if and only if it is a solution of
\mbox{\rm(\ref{eq1.1})} in the sense of Definition
\mbox{\rm\ref{def3.2}}.
\end{proposition}
\begin{proof}
See \cite[Proposition 3.12]{K:CVPDE}.
\end{proof}

In what follows for a function $u$ on $E$ and a measure $\mu$ on $E$, we set
\[
\langle\mu,u\rangle=\int_Eu(x)\mu(dx)
\]
whenever the integral is well defined, and for $k\ge0$,  we write
\[
T_ku(x)=\max\{\min\{u(x),k\},-k\},\quad x\in E.
\]

\begin{remark}
\label{rem2.2} (i) By \cite[Theorem 3.7]{K:CVPDE}, if $u$ is a
solution of (\ref{eq1.1}) then  $u$ is quasi-continuous.
\smallskip\\
(ii) Let $u$ be a solution of (\ref{eq1.1}) with $\mu\in\MM_b$. If
$f_u\in L^1(E;m)$ then by \cite[Theorem 3.3]{K:CVPDE}, $T_ku\in
D_e(\EE)$ for every $k\ge0$.  If, in addition, $m(E)<\infty$ or
$\EE$ satisfies Poincar\'e type inequality then $T_ku\in D(\EE)$
for $k\ge0$ (see \cite[Remark 3.4]{K:CVPDE}).
\end{remark}

In closing this section we recall yet another concept of solutions
introduced in \cite{K:CVPDE}.

We say that $u:E\rightarrow\BR\cup\{-\infty,+\infty\}$ is a
solution of (\ref{eq1.1}) in the sense of Stampacchia if for every
$v\in\BB(E)$ such that $\langle|\mu|,R|v|\rangle<\infty$ the integrals
$(u,v)$, $f_u\cdot m,Rv)$ are finite and
\[
(u,v)=(f_u,Rv)+\langle \mu,Rv\rangle.
\]
By \cite[Proposition 4.12]{K:CVPDE}, if $\mu\in\BM$, then $u$ is a
solution of (\ref{eq1.1}) in the sense of Stampacchia if and only
if it is a solution of (\ref{eq1.1}) in the sense of Definition
\ref{def3.1}.

\section{Renormalized solutions} \label{sec4}

As in Section \ref{sec3}, in this section we assume that
$(\EE,D(\EE))$ is transient and \mbox{\rm(ACR)} is satisfied. As
for the right-hand side of (\ref{eq1.1}), we restrict our
considerations to bounded measures.

The following definition
extends \cite[Definition 3.1]{KR:NoD} to possibly nonsmooth
measures.

\begin{definition}
\label{def4.1}
Let $\mu\in\MM_{b}(E)$. We say that
$u:E\rightarrow\BR\cup\{-\infty,+\infty\}$ is a renormalized
solution of (\ref{eq1.1}) if
\begin{enumerate}
\item[(a)]$u$ is quasi-continuous, $f_u\in L^1(E;m)$
and $T_ku\in D_e(\EE)$ for every $k\ge0$,
\item[(b)] there exists a sequence
$\{\nu_k\}\subset\MM_{0,b}(E)$ such that $R\nu_k\rightarrow
R\mu_c$ q.e. as $k\rightarrow\infty$, and for every $k\in\BN$ and
every bounded $v\in D_e(\EE)$,
\begin{equation}
\label{eq3.20} \EE(T_ku,v) =\langle f_u\cdot m+\mu_d,\tilde
v\rangle + \langle\nu_k,\tilde v\rangle.
\end{equation}
\end{enumerate}
\end{definition}

Note that in the case of local operators, the above definition  is essentially \cite[Definition 2.29]{DMOP}. A similar in spirit definition of  renormalized solutions of parabolic equations with local Leray-Lions type operators is considered
in \cite[Definition 4.1]{PPP} (in case $\mu_c=0$) and  \cite[Definition 3]{PP} (in the case of general bounded measures).

In  case  $\mu_c=0$, Definition \ref{def4.1} reduces to
\cite[Definition 3.1]{KR:NoD} with the exception that in
\cite{KR:NoD} in condition (b) it is required that
$\|\nu_k\|_{TV}\rightarrow0$. Note that in the case where $\mu_c\neq0$
the condition $R\nu_k\rightarrow R\mu_c$ q.e.
cannot be replaced by the  condition $\|\nu_k-\mu_c\|_{TV}\rightarrow0$ because the limit, in the total variation norm,  of diffuse measures is diffuse. Also, if $\mu_c\neq0$, then $\|\nu_k\|_{TV}\nrightarrow0$, because by \cite[Lemma 2.5]{KR:NoD2}, if $\|\nu_k\|_{TV}\rightarrow0$, then there is a subsequence $\{\nu_{k'}\}$ such that $R\nu_{k'}\rightarrow0$ q.e.
We see that the difference between the case $\mu_c=0$ and $\mu_c\neq0$ is  quite similar to that for parabolic equations considered in \cite{PPP,PP} (cf. \cite[Definition 4.1]{PPP} and \cite[Definition 3]{PP}).

\begin{remark}
(i) Let $E\subset \BR^d$ be a bounded domain, and let $L$ be the
Laplace operator $\Delta$ on $E$ with zero  boundary conditions.
By \cite[Remark 4.15]{K:CVPDE}, if $u$ is a renormalized solution
of (\ref{eq1.1}), then $u$ is a weak solution in the sense of
\cite{BMP}.
\smallskip\\
(ii) Let $\alpha\in(0,2]$, $E\subset \BR^d$ be a bounded domain,
and let $L$ be the fractional Laplacian  $\Delta^{\alpha/2}$ on
$E$ with zero boundary conditions. By \cite[Remark 4.13]{K:CVPDE},
if $u$ is a renormalized solution of (\ref{eq1.1}), then $u$ is a
solution of (\ref{eq1.1}) in the sense of \cite[Definition
1.1]{CV}.
\end{remark}

The following lemma is a modification of \cite[Lemma 5.4]{KR:JFA}.
As compared with \cite[Lemma 5.4]{KR:JFA}, we do not assume that
$\mu$ is smooth, but we additionally  require that the form
satisfies (ACT).

\begin{lemma}
\label{lem4.2} Assume that $\nu\in\BM\cap S^+$, $\mu\in\MM^+_{b}$.
If $R\nu\le R\mu$ $m$-a.e. then $\nu\in\MM^+_{0,b}$. In fact,
$\|\nu\|_{TV}\le\|\mu\|_{TV}$.
\end{lemma}
\begin{proof}
Set $ g_n=n(1-nR_n1)$. Then by the resolvent identity,
\[
Rg_n=nR_n1\le 1,\quad n\ge1.
\]
Since by \cite[Chapter II, Proposition (2.2)]{BG} the constant
function 1 is excessive relative to $\BBM$, $g_n\ge 0$ and, by
\cite[Chapter II, Proposition (2.3)]{BG}, $Rg_n\nearrow 1$. Since
the resolvent density $r$ is symmetric, applying Fubini's theorem
we get
\begin{align*}
\langle\mu,Rg_n\rangle
&=\int_E\!\Big(\int_Er(x,y)g_n(y)\,dy\Big)\,\mu(dx)\\
&=\int_E\!\Big(\int_Er(y,x)\,\mu(dx)\Big)\,g_n(y)\,dy =\langle
g_n,R\mu\rangle.
\end{align*}
Likewise, $\langle\nu,Rg_n\rangle=\langle g_n,R\nu\rangle$. Since
$R\nu\le R\mu$ $m$-a.e., it follows from the above that
\[
\langle\mu,Rg_n\rangle\ge\langle\nu,Rg_n\rangle,\quad n\ge1.
\]
Therefore
\[
\|\nu\|_{TV}=\lim_{n\rightarrow \infty}\langle Rg_n,\nu\rangle \le
\lim_{n\rightarrow \infty}\langle Rg_n,\mu\rangle=\|\mu\|_{TV},
\]
which proves the lemma.
\end{proof}

\begin{theorem}
\label{th4.3} Let $\mu\in\MM_{b}$.
\begin{enumerate}
\item[\rm(i)]If $u$ is a probabilistic solution of \mbox{\rm(\ref{eq1.1})}
and $f_u\in L^1(E;m)$ then $u$ is a renormalized solution of
\mbox{\rm(\ref{eq1.1})}.
\item[\rm(ii)]If $u$ is a renormalized solution of
\mbox{\rm(\ref{eq1.1})} then $u$ is a probabilistic solution of
\mbox{\rm(\ref{eq1.1})}.
\end{enumerate}
\end{theorem}
\begin{proof} (i) Let $Y_t=u(X_t)$, $t\ge0$. By (\ref{eq2.2}), for q.e. $x\in E$,
\begin{equation}
\label{eq4.3} Y_t=Y_0-\int_0^tf_u(X_s)\,ds
-\int^t_0dA^{\mu_d}_s+\int_0^t\,dM_s, \quad t\ge 0,\quad
P_x\mbox{-a.s.}
\end{equation}
By It\^o's formula for convex functions (see, e.g., \cite[Theorem
IV.66]{Pr}),
\begin{equation}
\label{eq2.4}
u^+(X_t)-u^+(X_0)=\int^t_{0}\fch_{\{Y_{s-}>0\}}\,dY_s+A^1_t,\quad
t\ge0,
\end{equation}
\begin{equation}
\label{eq4.4}
u^-(X_t)-u^-(X_0)=-\int^t_{0}\fch_{\{Y_{s-}\le0\}}\,dY_s+A^2_t,\quad
t\ge0
\end{equation}
for some increasing processes $A^1,A^2$. By \cite[Remark 3.10]{K:CVPDE}, there is a reducing sequence $\{\tau_k\}$ for $u$. Since $M$ is a local martingale under $P_x$ for q.e. $x\in E$,  for q.e. $x\in E$ there exists a sequence of stopping times  $\{\sigma_n\}$ (possibly depending on $x$) such that $E_x\int^{t\wedge\sigma_n}_0\fch_{\{Y_{s-}\le0\}}\,dM_s=0$, $t\ge0$, $n\ge1$.
Therefore, by (\ref{eq4.3}) and (\ref{eq2.4}),
\[
E_xA^1_{\tau_k\wedge\sigma_n}
=E_xu^+(X_{\tau_k\wedge\sigma_n})-u^+(x)+E_x\int^{\tau_k\wedge\sigma_n}_0\fch_{\{Y_{s-}>0\}}
(f_u(X_s)\,ds+dA^{\mu_d}_s)
\]
for all $k,n\ge1$. Letting $n\rightarrow\infty$ we get
\[
E_xA^1_{\tau_k}
=E_xu^+(X_{\tau_k})-u^+(x)+E_x\int^{\tau_k}_0\fch_{\{Y_{s-}>0\}}
(f_u(X_s)\,ds+dA^{\mu_d}_s).
\]
Similarly, by  (\ref{eq4.3}) and (\ref{eq4.4}),
\[
E_xA^2_{\tau_k} =E_xu^-(X_{\tau_k})-u^-(x)
-E_x\int^{\tau_k}_0\fch_{\{Y_{s-}\le0\}}
(f_u(X_s)\,ds+dA^{\mu_d}_s).
\]
Letting $k\rightarrow\infty$ in the above two equalities and using
(\ref{eq3.4}) shows that for q.e. $x\in E$,
\[
E_xA^1_{\zeta}\le R\mu^+_c(x)
+E_x\int^{\zeta}_0(|f_u(X_t)|\,ds+dA^{|\mu_d|}_t) =
R\mu^+_c(x)+R(|f_u|\cdot m+|\mu_d|)(x),
\]
\[
E_xA^2_{\zeta}\le R\mu^-_c(x)
+E_x\int^{\zeta}_0(|f_u(X_t)|\,ds+dA^{|\mu_d|}_t) =
R\mu^-_c(x)+R(|f_u|\cdot m+|\mu_d|)(x).
\]
By this and Proposition \ref{prop2.1},
$E_x(A^1_{\zeta}+A^2_{\zeta})<+\infty$ for q.e. $x\in E$.
Therefore by \cite[Theorem A.3.16]{FOT} there exists positive AFs
of $B^1,B^2$ of $\BBM$  such that $B^i$, $i=1,2$, is a compensator
of $A^i$ under $P_x$ for q.e. $x\in E$.  The processes $B^{1},
B^{2}$ are increasing, because  $A^{1}$ and $A^{2}$ are
increasing.  Since by \cite[Theorem A.3.2]{FOT} the process $X$
has no predictable jumps, it follows from \cite[Theorem
A.3.5]{FOT} that $B^{1}, B^{2}$ are continuous. Thus $B^{1},
B^{2}$ are increasing continuous AFs of $\BBM$ such that $A^{i}-B^{i}$,
$i=1,2$, is a martingale under $P_x$ for q.e. $x\in E$. Let
$b^i\in S$, $i=1,2$, denote the measure corresponding to $B^i$ in
the Revuz sense. Then, by (\ref{eq2.1}),
\[
Rb^i(x)=E_xB^i_{\zeta}=E_xA^i_{\zeta}<+\infty,\quad i=1,2,
\]
for q.e. $x\in E$. From this and Lemma \ref{lem4.2} it follows
that $b^1,b^2\in\MM_{0,b}$. By It\^o's formula, for $k>0$ we have
\begin{equation}
\label{eq2.6} (u^+\wedge k)(X_t)-(u^+\wedge k)(X_0)
=\int^t_0\fch_{\{u^+(X_{s-})\le k\}}\,du^+(X_s)-A^{1,k}_t,\quad
t\ge0,
\end{equation}
\begin{equation}
\label{eq4.7} (u^-\wedge k)(X_t)-(u^-\wedge k)(X_0)
=\int^t_0\fch_{\{u^-(X_{s-})\le k\}}\,du^-(X_s)-A^{2,k}_t,\quad
t\ge0,
\end{equation}
for some increasing processes $A^{1,k}, A^{2,k}$. By (\ref{eq2.4})
and (\ref{eq2.6}),
\[
E_xA^{1,k}_t\le u^+(x)\wedge k +E_x\int^t_0\fch_{\{u^+(X_{s-})\le
k\}}\fch_{\{Y_{s-}>0\}}\,dY_s +E_x\int^t_0\fch_{\{u^+(X_{s-})\le
k\}}\,dA^1_s
\]
whereas by (\ref{eq4.4}) and (\ref{eq4.7}),
\[
E_xA^{2,k}_t\le u^-(x)\wedge k -E_x\int^t_0\fch_{\{u^-(X_{s-})\le
k\}}\fch_{\{Y_{s-}\le0\}}\,dY_s +E_x\int^t_0\fch_{\{u^-(X_{s-})\le
k\}}\,dA^2_s.
\]
By the above two inequalities,
\[
E_x(A^{1,k}_{\zeta}+A^{2,k}_{\zeta})\le u^+(x)\wedge k
+u^-(x)\wedge k +R(|f_u|\cdot m+|\mu_d|)(x) +R(b^1+b^2)(x).
\]
Hence $E_x(A^{1,k}_{\zeta}+A^{2,k}_{\zeta})<+\infty$ for q.e.
$x\in E$. Let $B^{1,k},B^{2,k}$ be positive AFs of $\BBM$  such
that $B^{i,k}$, $i=1,2$, is a compensator of $A^{i,k}$ under $P_x$
for q.e. $x\in E$. As in case of $B^{1}, B^{2}$, we show that
$B^{1,k}, B^{2,k}$ increasing continuous AFs of $\BBM$ such that
$A^{i,k}-B^{i,k}$, $i=1,2$, is a martingale under $P_x$ for q.e.
$x\in E$. Let $b^{i,k}\in S$, $i=1,2$, denote the measure
corresponding to $B^{i,k}$ in the Revuz sense. Then
$R(b^{1,k}+b^{2,k})(x)=E_x(A^{1,k}_{\zeta}+A^{2,k}_{\zeta})<+\infty$
for q.e. $x\in E$, and hence, by Lemma \ref{lem4.2}, that
$b^{1,k},b^{2,k}\in\MM_{0,b}$. Let $Y^k_t=T_ku(X_t)$. Since
$T_ku=(u^+\wedge k)-(u^-\wedge k)$, from
(\ref{eq4.3})--(\ref{eq4.7}) we get
\begin{align}
\label{eq4.8} Y^k_t-Y^k_0&=-\int^t_0\fch_{\{-k\le Y_{s-}\le
k\}}(f_u(X_s)\,ds+dA^{\mu_d}_s)
-B^{1,k}_t\nonumber\\
&\quad+\int^t_0\fch_{\{u^+(X_{s})\le k\}}\,dB^1_s
+B^{2,k}_t-\int^t_0\fch_{\{u^-(X_{s})\le k\}}\,dB^2_s +M^k_t,
\end{align}
where
\begin{align*}
M^k_t&=\int^t_0\fch_{\{-k\le Y_{s-}\le k\}}\,dM_s
-(A^{1,k}_t-B^{1,k}_t) +(A^{2,k}_t-B^{2,k}_t)\\
&\quad +\int^t_0\fch_{\{u^+(X_{s-})\le k\}}\,d(A^1_s-B^1_s)
-\int^t_0\fch_{\{u^-(X_{s-})\le k\}}\,d(A^2_s-B^2_s).
\end{align*}
Since $M^k$ is a martingale under $P_x$ for q.e. $x\in E$, from
(\ref{eq4.8}) it follows that for q.e. $x\in E$,
\begin{align*}
T_ku(x)&=E_xT_k(X_t)+E_x\int^t_0\fch_{\{-k\le Y_{s-}\le
k\}}(f_u(X_s)\,ds+dA^{\mu_d}_s)\\
&\quad+E_xB^{1,k}_t-E_x\int^t_0\fch_{\{u^+(X_{s})\le k\}}\,dB^1_s
-E_xB^{2,k}_t+E_x\int^t_0\fch_{\{u^-(X_{s})\le k\}}\,dB^2_s.
\end{align*}
Since $T_ku(X_t)\rightarrow0$ $P_x$-a.s. as $t\rightarrow\infty$,
$E_xT_ku(X_t)\rightarrow0$ by the Lebesgue dominated convergence
theorem. Therefore from the above equality it follows that
\[
T_ku(x)=R(\fch_{\{-k\le u\le k\}}(f_u\cdot m+\mu_d))
+R(b^{1,k}-\fch_{\{u^+\le k\}}b^1)-R(b^{2,k}-\fch_{\{u^-\le k\}}
b^2).
\]
Set
\[
\nu_k=\fch_{\{u\notin[-k,k]\}}(f_u\cdot m+\mu_d)
+b^{1,k}-\fch_{\{u^+\le k\}}b^1-b^{2,k}+\fch_{\{u^-\le k\}}b^2.
\]
Then $\nu_k\in\MM_{0,b}$ and for q.e. $x\in E$,
\begin{equation}
\label{eq2.3} T_ku(x)=R(f_u\cdot m+\mu_d)(x) +R\nu_k(x).
\end{equation}
On the other hand, by  Proposition \ref{prop3.4},
$u(x)=R(f_u\cdot m+\mu_d)(x) +R\mu_c(x)$ for q.e. $x\in E$. Hence
$R\nu_k(x)\rightarrow R\mu_c(x) $ for q.e. $x\in E$. By Remark
\ref{rem2.2}(ii),  $T_ku\in D_e(\EE)$. Finally, since
$T_ku=R\lambda_k$ with $\lambda_k=f_u\cdot m+\mu_d
+\nu_k\in\MM_{0,b}$, repeating step by step the reasoning
following \cite[(3.14)]{KR:NoD} shows  that $T_ku$ satisfies
(\ref{eq3.20}), which completes the proof of (i).

(ii) Assume that $u$ is a renormalized solution of (\ref{eq1.1}).
Then $T_ku$ is a solution in the sense of duality of the linear
equation
\[
-L(T_ku)=f_u+\mu_d+\nu_k,
\]
and hence $T_ku$ is a probabilistic solution of the above equation
(see the arguments in \cite[p. 1924]{KR:NoD}). Hence
\[
T_ku(x)=E_x\Big(\int^{\zeta}_0(f_u(X_t)\,dt
+dA^{\mu_d}_t)+\int^{\zeta}_0dA^{\nu_k}_t\Big)=R(f_u\cdot
m+\mu_d)(x)+R\nu_k(x)
\]
for q.e. $x\in E$. Since $R\nu_k\rightarrow R\mu_c$ q.e., letting
$k\rightarrow\infty$ in the above equation we see that
(\ref{eq1.4}) is satisfied for q.e. $x\in E$, i.e. $u$ is a
solution of (\ref{eq1.1}) in the sense of Definition \ref{def3.1}.
By this and Proposition \ref{prop3.4}, $u$ is a probabilistic
solution of (\ref{eq1.1}).
\end{proof}

Note that by Proposition \ref{prop3.4}, in the formulation of
Theorem \ref{th4.3} we may replace ``probabilistic solution" by
``solution in the sense of Definition \ref{def3.1}", while by
\cite[Proposition 4.12]{K:CVPDE} we may replace ``probabilistic
solution" by ``solutions in the sense of Stampacchia".

By Theorem \ref{th4.3}, a probabilistic solution $u$ is a
renormalized solution once we know that $f_u\in L^1(E;m)$. We
close this section with describing some  interesting situations in
which  this condition holds true.

\begin{proposition}
\label{prop4.4} Let $\mu\in\MM_b$ and let
$f:E\times\BR\rightarrow\BR$ be a measurable function such that
$f(\cdot,0)\in L^1(E;m)$ and for every $x\in E$ the mapping
$\BR\ni y\mapsto f(x,y)$ is continuous and nonincreasing. If $u$
is a probabilistic solution of \mbox{\rm(\ref{eq1.1})}  then
$f_u\in L^1(E;m)$.
\end{proposition}
\begin{proof}
See \cite[Proposition 4.8]{K:CVPDE}.
\end{proof}

Following \cite{BMP,K:CVPDE}  we call $\mu\in\BM$ a good measure
(relative to $L$ and $f$) if there exists a probabilistic solution
of (\ref{eq1.1}).

\begin{proposition}
\label{prop4.5} Assume that $f$ satisfies the assumptions of
Proposition \ref{prop4.4} and $\mu\in\BM$ is good relative to $L$
and $f$. Then there exists a unique renormalized solution of
\mbox{\rm(\ref{eq1.1})}. Moreover, for every $k\ge0$,
\begin{equation}
\label{eq4.20} \EE(T_ku,T_ku)\le
k(\|\mu\|_{TV}+\|f_u\|_{L^1(E;m)}),
\end{equation}
\begin{equation}
\label{eq4.21} \|f_u\|_{L^1(E;m)}\le
2\|f(\cdot,0)\|_{L^1(E;m)}+\|\mu\|_{TV}.
\end{equation}
\end{proposition}
\begin{proof}
The existence of a solution follows immediately from Theorem
\ref{th4.3}(i) and Proposition \ref{prop4.4}. Uniqueness follows
from Theorem \ref{th4.3}(ii) and \cite[Corollary 4.3]{K:CVPDE}.
Estimate (\ref{eq4.20}) follows from \cite[Theorem 3.3]{K:CVPDE},
whereas (\ref{eq4.21}) from \cite[Proposition 4.8]{K:CVPDE}.
\end{proof}

The following remark shows that the monotonicity assumption
imposed on $f$ in Propositions \ref{prop4.4} and \ref{prop4.5} can
be relaxed in case $\mu$ is nonnegative.

\begin{remark}
(i) Assume that $\mu\in\BM$ is nonnegative and $f$ satisfies the
following ``sign condition": for every $x\in E$,
\begin{equation}
\label{eq4.10} yf(x,y)\le0, \quad y\in\BR.
\end{equation}
Then if $u$ is a probabilistic solution of (\ref{eq1.1}), then
$u\ge0$ q.e. To see this, let us consider a reducing sequence
$\{\tau_k\}$ for $u$. Then by (\ref{eq4.3}), (\ref{eq4.4}) and
It\^o's formula for convex functions (see \cite[Theorem
IV.66]{Pr}), for q.e. $x\in E$ we have
\[
u^-(x)=E_xu^-(X_{\tau_k})-\int^{\tau_k}_0\fch_{\{Y_{s-}\le0\}}f(X_s,Y_{s})\,ds
-\int^{\tau_k}_0\fch_{\{Y_{s-}\le0\}}\,dA^{\mu_d}_s-E_xA^2_{\tau_k}.
\]
Since $\mu\ge0$, $\mu_d\ge0$ and $\mu_c\ge0$. In particular,
$A^{\mu_d}$ is increasing. Since $A^2$ is also increasing and $f$
satisfies (\ref{eq4.10}), it follows that $ u^-(x)\le
E_xu^-(X_{\tau_k})$. By this and  (\ref{eq3.4}),
$u(x)\le\lim\sup_{k\rightarrow\infty}E_xu^-(X_{\tau_k})=R\mu^-_c(x)=0$
for q.e. $x\in E$.\smallskip\\
(ii) Obviously  (\ref{eq4.10}) is satisfied if $f(x,0)=0$ and $f$
is nonincreasing. Therefore if $\mu$ in Proposition \ref{prop4.4}
is nonnegative, then without loss of generality we may assume that
$f(\cdot,y)=0$ for $y\le0$, i.e. $f$ satisfies the condition
imposed on $f$ in \cite{BMP} (see \cite[Remark 1]{BMP}) and in
\cite[Section 5]{K:CVPDE}.
\smallskip\\
(iii) If $f$ satisfies (\ref{eq4.10}) and $\mu\in\MM^+_b$ is good
(relative to $L$ and $f$), then $f_u\in L^1(E;m)$, and hence there
exists a renormalized solution of (\ref{eq1.1}). Indeed, if
$\mu\ge0$ then by part (i), $u\ge0$ q.e., and consequently
$Rf_u+R\mu\ge0$ q.e. and  $f_u\le0$. Hence $0\le R(-f_u)=-Rf_u\le
R\mu$ q.e. By this and Lemma \ref{lem4.2}, $-f_u\cdot
m\in\MM^+_b$,  so $f_u\in L^1(E;m)$.
\end{remark}

The problem of existence of solutions of (\ref{eq1.1}) for $f$
satisfying the assumptions of Proposition \ref{prop4.4} (or more
general ``sign condition" (\ref{eq4.10})) and the related problem
of  characterizing the set of good measures are very subtle, and
are beyond the scope of the present paper. For many positive
results in this direction in the case where $A$ is the Laplace
operator we defer the reader to \cite{BMP,Po}. Interesting
existence and uniqueness results for equations  involving the
fractional Laplace operator are to be found in  \cite{K:CVPDE,CV}.
\smallskip\\
{\large\bf Acknowledgements}
\smallskip\\
This work was supported by the Polish National Science Centre under Grant\\
2012/07/B/ST1/03508.

\end{document}